\documentclass[a4paper, 12pt]{amsart}

\usepackage{amsmath}
\usepackage{amssymb}
\usepackage{amsthm}
\usepackage[usenames,dvipsnames]{xcolor}

\usepackage[final]{showkeys}
\usepackage{url}
\usepackage{ mathrsfs }
\definecolor{DeepBlue}{rgb}{0,0,.7}
\usepackage[dvipdfm, setpagesize=false]{hyperref}
\hypersetup{colorlinks=true, citecolor=Sepia, linkcolor=DeepBlue, urlcolor=DeepBlue, pdfborder= 0 0 0}
\hypersetup{setpagesize=false}
\usepackage[all]{hypcap}
\usepackage[all]{xy}

\newcommand{\HH}{\mathbb{H}}

\newcommand{\R}{\mathbb{R}}

\newcommand{\Ob}{\mathcal{O}}
\newcommand{\C}{\mathbb{C}}

\newcommand{\w}[1]{\widetilde{#1}}
\newcommand{\id}{\mathrm{id}}

\newcommand{\lN}[1]{\lambda_{\eta(#1)}}

\newcommand{\PK}{P_{\sigma_{\tau(m),k}}(t)}
\newcommand{\PG}{P_{\sigma_{\tau(m),k}}^\gamma(t)}

\newcommand{\bs}{\backslash}

\newcommand{\Spin}{\text{\textnormal{Spin}}}
\newcommand{\End}{\text{End}}
\newcommand{\tr}{\text{\textnormal{tr}}\thinspace}
\newcommand{\diag}{\text{\textnormal{diag}}\thinspace}
\newcommand{\Tr}{\text{\textnormal{Tr}}\thinspace}
\newcommand{\dom}{\text{\textnormal{dom}}\thinspace}
\newcommand{\spec}{\text{spec}}
\newcommand{\vol}{\text{\textnormal{vol}}}

\newcommand{\ad}{\text{\textnormal{ad}}}

\newcommand{\Ad}{\text{Ad}}
\newcommand{\RRe}{\text{Re}}

\newcommand{\N}{\mathbb{N}}
\usepackage{accents}
\newlength{\dtildeheight}

\newtheorem{Th}{Theorem}[section]
\newtheorem{Lemma}[Th]{Lemma}
\newtheorem{Proposition}[Th]{Proposition}
\newtheorem{Corollary}[Th]{Corollary}
\theoremstyle{definition}
\newtheorem{lem}{Lemma}[section]
\newtheorem{remark}[lem]{Remark}
\newtheorem{definition}{Definition}[section]

\newtheorem{exmp}{Example}[section]
\begin{document}

\title{On the asymptotics of the analytic torsion for compact hyperbolic orbifolds}
\author{Ksenia Fedosova}
%\thanks{...}
\address{Max Planck Institute for Mathematics\\ Vivatgasse 7\\ 53111 Bonn \\ Germany}

\email{fedosova@math.uni-bonn.de}

%\date{20.09.09}

\begin{abstract}
We study the analytic torsion of odd-dimensional hyperbolic  orbifolds~$\Gamma \bs \HH^{2n+1}$, depending on a representation of $\Gamma$. Our main goal is to understand the asymptotic behavior of the analytic torsion with respect to sequences of representations associated to rays of highest weights.
\end{abstract}

\maketitle

\setcounter{tocdepth}{1}

\section{Introduction}
The Cheeger-M\"uller theorem asserts the equality of the Reidemeister torsion and the analytic torsion for a smooth odd-dimensional manifold. The Reidemeister torsion is a purely combinatorial invariant, whereas the analytic torsion is defined in terms of the spectrum of Laplace operators.  Because these invariants have very different natures, their equality has applications in topology, number theory, and mathematical physics. One such application is to study the growth of torsion in the cohomology of an arithmetic group $\Gamma$. This was carried out in \cite{MM,MP111} when $\Gamma$ is cocompact and without elliptic elements.

One of the steps was to prove the exponential growth of the analytic torsion under certain rays of representations \cite{Mu2,MP2}. 
Our goal is to extend the results in \cite{MP2,Mu2} to a cocompact group $\Gamma \subset \Spin(1,2n+1)$ that may have elliptic elements, because prohibiting such elements excludes a wide range of arithmetic groups. Assuming $\Gamma \subset \Spin(1,2n+1)$, so that $\Gamma \bs \Spin(1,2n+1) / \Spin(2n+1)$ is a compact odd-dimensional hyperbolic orbifold, we define the analytic torsion of such orbifold through the zeta function of the Laplace-de Rham operator in a similar way as it is done for manifolds; see Definition \ref{senbernary}. Our main result is the following theorem.
\begin{Th}\label{adacupcakeprincess}
Let $\Ob = \Gamma \bs \HH^{2n+1}$ be a compact hyperbolic  orbifold. For $m \in \N$, let $\tau(m)$ be the finite-dimensional irreducible representation of $\Spin(1,2n+1)$ from Definition \ref{mimirep}
and $\tau'(m)$ be the restriction of $\tau(m)$ to $\Gamma$.
Let $E_{\tau(m)}\to \Ob$ be the associated flat vector orbibundle, and denote by $T_\Ob (\tau(m))$
 its analytic torsion. Then there exists $C > 0$ such that  
$$\log T_\Ob (\tau(m)) =  PI(m) + PE(m) + O(e^{-C  m}), \quad m \to \infty.$$
Above, $PI(m)$ is a polynomial in $m$ of degree $\frac{n^2+n+2}{2}$ and $PE(m)$ is a pseudo-polynomial in $m$
of degree $\frac{d^2+d+2}{2}$ with $d$ being the 
maximal dimension of a singular stratum of $\Ob$. The pseudopolynomial $PE(m)$ is a sum of type
$\sum_{j=0}^{\frac{d^2+d+2}{2}} \sum_{k=0}^K C_j m^j e^{i m \phi_{j,k}},$
where $C_j, \phi_{j,k} \in \R$, $K \in \N$ are  constants depending on the elliptic elements of $\Gamma$.
\end{Th}

\begin{remark}
The term $PE(m)$ does not appear when $\Ob$ is a manifold, compare \cite{MP2}.
\end{remark}

The proof of Theorem \ref{adacupcakeprincess} is based on \cite{MP2} but requires ingredients that we soon specify. Let $\tau(m)$ and $E_{\tau(m)} \to \Ob$ be as above; the vector orbibundle $E_{\tau(m)} \to \Ob$ can be equipped with a canonical Hermitian fibre metric \cite[Proposition 3.1]{MaMu}. Let $\Delta_p(\tau(m))$ be the Laplace operator on $E_{\tau(m)}$-valued $p$-forms with respect to the metric on $E_{\tau(m)}$ and the hyperbolic metric on~$\Ob$. Denote 
$$ K(t, \tau(m)):=\sum_{p=0}^{2n+1} (-1)^p p\, \Tr \left( e^{-t \Delta_p(\tau(m))} \right),$$
then the analytic torsion is given by
\begin{equation}\label{lordofmercy}
\log T_{\Ob}(\tau(m))=\frac{1}{2} \frac{d}{ds} \left.\left(   \frac{1}{\Gamma(s)} \int_0^\infty t^{s-1} K(t,\tau) dt \right)\right|_{s=0}.
\end{equation}
For further estimates we need the following theorem:
\begin{Th}\label{cuppy}
Let $(\Ob, g)$ be a compact Riemannian orbifold, not necessarily hyperbolic, $E \to \Ob$ an associated flat orbibundle and $\Delta_k$ the Laplacian acting on $E$-valued $k$-form. Assume that $\Ob$ and its singular strata are odd dimensional and that the cohomology $H^j(\Ob;E)$ vanishes. Then the analytic torsion $T_\Ob(h,g)$ does not depend on a smooth deformation of the metric $g$.
\end{Th}
Theorem \ref{cuppy} is of interest on its own: it is an important problem to understand the relation between the analytic and the Reidemeister torsions for orbifolds. If they were equal, Theorem~\ref{adacupcakeprincess} would imply the exponential growth of torsion in the cohomology of cocompact arithmetic groups \cite{MP111, MM}.
In turn, any reasonable relation between the analytic and the Reidemeister torsions can be expected to imply that the former does not depend on the variation of the metric; this is shown in Theorem \ref{cuppy} under certain restrictions. 

We will now describe a rough plan of the proof of Theorem~\ref{adacupcakeprincess}. As a simple corollary of Theorem~\ref{cuppy}, we can scale the metric $g$ on $\Ob$ and hence replace $\Delta_p(\tau(m))$ by $\frac{1}{m} \Delta_p(\tau(m))$ in (\ref{lordofmercy}). Splitting the integral into the integral over $[0,1)$ and the integral over $[1,\infty)$, we obtain
$$ \log T_{\Ob}(\tau(m))=
  \left. \frac{1}{2} \frac{d}{ds} \left( \frac{1}{\Gamma(s)} \int_0^1 t^{s-1} K\left(\frac{t}{m}, \tau(m)\right) d t  \right) \right|_{s=0} + \frac{1}{2}  \int_1^\infty t^{s-1} K(t,\tau) dt.$$
It follows from \cite{MP2} that the second term is $O(e^{-m/8})$ as $m \to \infty$; to estimate the first term we use the Selberg trace formula. For this we construct a smooth $K$-finite function $k_t^\tau(m)$ on $\Spin(1,2n+1)$ as in Subsection \ref{senbernary} such that
$$ K(t,\tau(m)) = \int_{\Gamma \bs \Spin(1,2n+1)} \sum_{\gamma \in \Gamma} k_t^{\tau(m)} (g^{-1} \gamma g) d \dot{g}.$$
By the Selberg trace formula for compact orbifolds, 
$$K(t,\tau(m)) = I(t, \tau(m)) + H(t, \tau(m)) +E(t, \tau(m)),$$
where $I(t,\tau(m))$, $H(t,\tau(m))$ and $E(t,\tau(m))$ are the contributions from the identity, hyperbolic and elliptic elements of $\Gamma$, respectively. Note $E(t,\tau(m))$ vanishes if $\Ob$ is a smooth manifold. It follows from \cite{MP2} that there  exist $C, c_1, c_2 > 0$ such that 
$$\left| H(t/m, \tau(m)) \right| \le C e^{-c_1 m} e^{-c_2 m}$$
for all $m \ge m_0$ and $0<t \le 1$. Recall that $I(t,\tau(m)) = \vol(\Ob) k_t^{\tau(m)} (1)$. We can switch back from $t/m$ to the variable $t$, then the identity contribution to $\log T_{\Ob}(\tau(m))$ is given by
$$ \frac{\vol(\Ob)}{2}  \frac{d}{ds} \left.\left(   \frac{1}{\Gamma(s)} \int_0^\infty t^{s-1} k_t^{\tau(m)} (1) dt \right)\right|_{s=0},$$
 As in \cite{MP2} we apply the Plancherel formula to $k_t^{\tau(m)}(1)$ and use the properties of Plancherel polynomials. We are left with the contribution from the elliptic elements. An important ingredint is now to apply the result of our previous paper \cite{Fe1}, which says that there exist polynomials $P^\gamma(i \nu)$ such that 
$$  E(t,\tau(m)) = \sum_{\{\gamma   \} \, \textnormal{elliptic}} \vol(\Gamma_\gamma \bs G_\gamma) \sum_{\sigma \in \widehat{M}} \int_\R P^\gamma_\sigma(i \lambda) \Theta_{\sigma,\lambda} (k_t^{\tau(m)}) d \lambda.$$
Above $\Theta_{\sigma,\lambda}$ is the character of the representation $\pi_{\sigma,\lambda}$ as in Subsection \ref{prince}. Using the properties of $P^\gamma_\sigma(i \nu)$, we obtain Theorem \ref{adacupcakeprincess}.

\subsection*{Structure of the article.} 
In Section \ref{howcan} we fix notation and collect some facts about the analytic torsion and the Selberg trace formula. In Sections \ref{searchedfor} and \ref{iamgone} we prove Theorem \ref{cuppy} and Theorem \ref{adacupcakeprincess}, respectively, using some technical lemmas postponed to Section \ref{sec:AppendC}.
In Appendix \ref{sect:AppendixA} we explain that our proof of Theorem \ref{cuppy} cannot be extended to manifolds with singularities that are not orbifolds.

\subsection*{Acknowledgement}
The present paper is a part of the author’s PhD thesis. She would like to thank her supervisor\, Werner M\"uller for his constant support. Also the author would like to thank Dmitry Tonkonog and Julie Rowlett for reading the draft and correcting minor mistakes.

\section{Preliminaries}\label{howcan}
The purpose of this section is to introduce various notation and definitions, most importantly we define certain rays of representations of a discrete group acting on the hyperbolic space, and the Ray-Singer analytic torsion. We also recall our main computational tool, the Selberg trace formula for hyperbolic orbifolds. 
\subsection{Lie groups.}\label{subsection1}
Let $G = \Spin(1,2n+1)$, $K=\Spin(2n+1)$. Let $G = NAK$ be the standard Iwasawa decomposition of $G$, hence for each $g \in G$ there are uniquely determined elements $n(g) \in N$, $a(g) \in A$, $\kappa(g) \in K$ such that
$$ g = n(g) a(g) \kappa(g).$$
Let $M$ be the centralizer of $A$ in $K$, then 
$$M = \Spin(2n).$$ 
Denote the Lie algebras of $G$, $K$, $A$, $M$ and $N$ by 
$\mathfrak{g}$,
$\mathfrak{k}$,
$\mathfrak{a}$,
$\mathfrak{m}$
and
$\mathfrak{n}$, respectively. Define the standard Cartan involution $\theta:\mathfrak{g} \to \mathfrak{g}$ by
$$\theta(Y) = - Y^t, \quad Y \in \mathfrak{g}.$$ Let $H: G \to \mathfrak{a}$ be defined by
$$ H(g) := \log a(g).$$
Equipped with a certain invariant metric, $G/K$ is isometric to the hyperbolic space $\HH^{2n+1}$; see \cite[p.~6]{MP111}.

\subsection{Hyperbolic orbifolds.} For the definition of orbifolds and orbibundles we refer to \cite{Dry} or \cite{Buc}, but for the reader's convenience we recall the following definition:
\begin{definition}\label{orbicharty}
An orbifold chart on a topological space $X$ consists of a  contractible open subset $\widetilde{U}$ of $\R^n$, a finite group $G_U$ 
acting on $\widetilde{U}$ by diffeomorphisms, and a mapping
$\pi_U$ from $\widetilde{U}$ onto an open subset $U$ of $X$ inducing a homeomorphism from the orbit space $G_U \bs \widetilde{U}$ onto $U$.
\end{definition}
Now restrict to hyperbolic orbifolds. Consider a discrete subgroup $\Gamma \subset G$, $G=\Spin(1,2n+1)$ such that $\Ob = \Gamma \bs \HH^{2n+1}$ is compact. 
It follows that all elements of $\Gamma$ are either hyperbolic or
elliptic. 
\begin{definition}
$\gamma \in \Gamma$ is called  elliptic if it is of finite order.
\end{definition}
\begin{definition}
$\gamma \in \Gamma$ is called hyperbolic if 
$$ l(\gamma) := \inf_{x \in \HH^{2n+1}} d(x, \gamma x) > 0,$$
where $d(x,y)$ denotes the hyperbolic distance between $x$ and $y$.
\end{definition}
Any elliptic element $\gamma$ is conjugate to an element in $K$, so without loss of generality we may assume~$\gamma$ is of the form: 
\begin{equation}\label{pelmeni}
\gamma = \diag(\overbrace{\left( \begin{smallmatrix}
1 & 0  \\
0 & 1  \end{smallmatrix} \right), \ldots, \left( \begin{smallmatrix}
1 & 0  \\
0 & 1  \end{smallmatrix} \right)}^d, \overbrace{R_{\phi_{d+1}}, \ldots, R_{\phi_{n+1}}}^{n-d+1}),
\end{equation}
 where $R_\phi = \left(      \begin{smallmatrix} \cos \phi & \sin \phi \\ -\sin \phi & \cos \phi \end{smallmatrix}  \right)$. 
 A proposition follows immediately:
 
 \begin{Proposition} \label{pony}
The fixed point set of the map $x \mapsto \gamma \cdot x$, where $x \in \HH^{2n+1}$ and $\gamma \in \Gamma$ is elliptic as in \ref{pelmeni}, is of dimension $2d-1$. 
\end{Proposition}

\subsection{Lie algebras}\label{liealg}
Denote by $E_{i,j}$ the matrix in $\mathfrak{g}$ whose $(i,j)$'th entry is 1 and the other entries are~0. Let
\begin{equation*}
\begin{gathered}
H_1 := E_{1,2} + E_{2,1},\\
H_j := i (E_{2j-1,2j} - E_{2j,2j-1}), \quad j = 2, \ldots, n+1.
\end{gathered}
\end{equation*}
Then $\mathfrak{a} = \R H_1$ and let  $\mathfrak{b} = i \R H_2 + \ldots + i \R H_{n+1}$ be the standard Cartan subalgebra of $\mathfrak{m}$. Moreover, $\mathfrak{h} = \mathfrak{a} \oplus \mathfrak{b}$ is a Cartan subalgebra of $\mathfrak{g}$. Define $e_i \in \mathfrak{h}_{\C}^*$ with $i = 1, \ldots, n+1$, by
\begin{equation}\label{makaronina}
 e_i(H_j) = \delta_{i,j}, \, 1\le i,j\le n+1.
\end{equation}
The sets of roots of $(\mathfrak{g}_\C, \mathfrak{h}_{\C})$ and   $(\mathfrak{m}_{\C}, \mathfrak{b}_{\C})$ are given by
\begin{equation}\label{nichegouzhenesvyazano}
\begin{gathered}
\Delta(\mathfrak{g}_\C, \mathfrak{h}_{\C}) = \{ \pm e_i \pm e_j, 1 \le i < j \le n+1\},\\
\Delta(\mathfrak{m}_{\C}, \mathfrak{b}_{\C}) = \{  \pm e_i \pm e_j, 2\le i < j \le n+1   \}.
\end{gathered}
\end{equation}
We fix the positive systems of roots by 
\begin{equation}\label{mojbambino}
\begin{gathered}
\Delta^+(\mathfrak{g}_\C, \mathfrak{h}_{\C}) = \{  e_i \pm e_j, 1 \le i < j \le n+1\},\\
\Delta^+(\mathfrak{m}_{\C}, \mathfrak{b}_{\C}) = \{   e_i \pm e_j, 2\le i < j \le n+1   \}.
\end{gathered}
\end{equation}
The half-sum of positive roots $\Delta^+(\mathfrak{m}_{\C}, \mathfrak{b}_{\C})$ equals 
\begin{equation}\label{halfsummy}
\rho_M = \sum_{j=2}^{n+1} \rho_j e_j,\quad \rho_j = n+1-j.
\end{equation}
Let $M'$ be the normalizer of $A$ in $K$ and let $W(A)=M'/M$ be the restricted Weyl group. It has order~2 and acts on finite-dimensional representations of $M$ \cite[p.~18]{Pf}. Denote by $w_0$ the non-identity element of $W(A)$.

\subsection{Principal series parametrization.}\label{prince} Let $\sigma: M \mapsto \End(V_\sigma)$ be a finite-dimensional irreducible representation of $M$. 
\begin{definition}
We define $\mathcal{H}^\sigma$ to be the space of measurable functions $f:K\mapsto V_\sigma$ such that
\begin{enumerate}
\item $f(mk)=\sigma(m) f(k)$ for all $k \in K$ and $m \in M$;
\item $\int_K || f(k)||^2 dk < \infty$.
\end{enumerate}
\end{definition}
Recall $H: G \to \mathfrak{a}$, $\kappa: G \to K$ is as in Subsection \ref{subsection1} and $e_1 \in \mathfrak{h}_\C^*$ are as in Subsection \ref{liealg}. For $\lambda \in \R$ define the representation $\pi_{\sigma, \lambda}$ of $G$ on $\mathcal{H}^\sigma$ by the following formula:
$$ \pi_{\sigma,\lambda} (g) f(k) := e^{(i \lambda e_1 + \rho) (H(kg))} f(\kappa(kg)),$$
where $f\in \mathcal{H}^\sigma$, $g\in G$.
\subsection{Representations.} Fix $\tau_1, \ldots, \tau_{n+1} \in \N,$ such that $\tau_1 \ge \tau_2 \ge \ldots \ge \tau_{n+1}$. Recall that $n~=~\frac{\dim(\Ob)-1}{2}$.
\begin{definition}\label{mimirep}
For $m \in \N$ denote by $\tau(m)$ the finite-dimensional representation of $G$ with highest weight 
$$(m + \tau_1)e_1 + \ldots + \ldots (m+\tau_{n+1}) e_{n+1}.$$
\end{definition}
This is a ray of representations which will be the focus of the article.
\begin{definition}\label{dertttt} Let $\tau$ be the finite-dimensional irreducible representation of $G$ with highest weight $\tau_1 e_1 + \ldots + \tau_{n+1} e_{n+1}$.
The denote by $\sigma_{\tau,k}$ be the representation of $M$ with  highest weight
$$\Lambda_{\sigma_{\tau,k}}:= (\tau_2 +1)e_2 + \ldots + (\tau_k +1)e_{k+1}+\tau_{k+2}e_{k+2} + \ldots + \tau_{n+1}e_{n+1}.$$
\end{definition}

\subsection{Analytic torsion.}\label{senbernary}
Let $E \to \Ob$ be a flat vector orbibundle over a compact Riemannian orbifold~$(\Ob, g)$; pick a Hermitian fiber metric $h$ in $E$. Denote by  $\Delta_p(h, g)$ the Laplacian acting on $E$-valued $p$-forms on $\Ob$. Note that $\spec(\Delta_p(h, g))$ is semi-bounded, hence we can define $e^{-t \Delta_p(h, g)}$ by a suitable Dunford integral as in \cite{Gil}.
\begin{definition}
The zeta function $\zeta_p(s;h,g)$ is
$$\zeta_p(s;h,g) := \frac{1}{\Gamma(s)} \int_0^\infty t^{s-1} \Tr (e^{-t \Delta_p(h, g)} (1-P)) dt,$$
where $P$ is the orthogonal projection to $\ker \Delta_p(h,g)$.
\end{definition}
It is holomorphic for $\RRe (s) < - (2n+1)/2$ and admits a meromorphic extension to $\C$.
It follows from \cite{Gil} and \cite{Fe1} that $\zeta_p(s;h,g)$ is regular at $0$.
\begin{definition}\label{analtorry}
The analytic torsion $T_{\Ob}(h,g)$ is
$$ \log T_\Ob(h,g) := \frac{1}{2} \sum_{p=1}^d (-1)^p p \frac{d}{ds} \left. \zeta_p(s;\rho) \right|_{s=0}.$$
\end{definition}
Now let $\Ob = \Gamma \bs \HH^{2n+1}$, let $\rho$ be a finite-dimensional representation of $\Gamma$ in a complex vector space~$V_\rho$ and  $E_\rho \to \Ob$ be the associated flat vector bundle. Pick a Hermitian fiber metric $h$ in $E_\rho$. 

Let us specify to the case when $\rho = \tau|_\Gamma$ is the restriction to $\Gamma$ of a finite-dimensional irreducible representation $\tau$ of $G$. In this case $E_\rho$ can be equipped with a distinguished metric which is unique up to scaling. Namely, $E_\rho$ is canonically isomorphic to the locally homogeneous orbibundle~$E_\tau$ associated to $\tau|_K$   \cite[Proposition 3.1]{MaMu}. Moreover, there exists a unique up to scaling inner product~$\langle \cdot, \cdot \rangle$ on $V_\rho$ such that
\begin{enumerate}
\item $\langle \tau(Y) u, v \rangle = - \langle u, \tau(Y) v \rangle, \quad  Y \in \mathfrak{k}$,
\item $\langle \tau(Y) u, v \rangle =  \langle u, \tau(Y) v \rangle, \quad  Y \in \mathfrak{p}$,
\end{enumerate}
for all $u,v \in V_\rho$. Note that $\tau|_K$ is unitary with respect to this inner product, hence it induces a unique up to scaling metric~$h$ on $E_\tau$.
\begin{definition}
Such a metric on $E_\tau$ is called admissible.
\end{definition}
From now on fix an admissible metric $h$. 
\begin{definition}\label{analtor}
In the above setting denote $T_\Ob(\rho):=T_\Ob(h,g)$; here $\rho$ indicates the dependence of the analytic torsion on the orbibundle $E_\rho$, and $T_\Ob(h,g)$ is from Definition \ref{analtorry}.
\end{definition} 
%Recall that there exists an isomorphism 
% \begin{equation}\label{isomorphismofspaces}
% \Lambda^p(\Ob, E_\tau) \cong C^\infty(\Gamma \bs G, \nu_p(\tau)),
% \end{equation}
% where 
%$$\nu_p(\tau):= \Lambda^p \Ad^* \otimes \tau: K \mapsto GL(\Lambda^p \mathfrak{p}^* \otimes V_\tau),$$
%$$ C^\infty(\Gamma \bs G, \nu):= \{ f: G \mapsto V_\nu : \,  f(\gamma g)=f(g), \, f(gk)=\nu(k^{-1}) f(g), \,  \forall g\in G, \forall k \in K, \forall \gamma \in \Gamma  \}.$$
%With respect to (\ref{isomorphismofspaces}) the Laplacian $\Delta_p(\tau)$ acts as
%\cite[(6.9)]{MaMu}
%$$ \Delta_p(\tau) f = -R(\Omega) f + \tau(\Omega) f, \, f \in C^\infty(\Gamma \bs G, %\nu_p(\tau)).$$
Let $\Delta_p(\tau)$ be the Hodge-Laplacian on $\Lambda^p(\Ob, E_\tau)$ with respect to an admissible metric in $E_\tau$. Denote 
\begin{equation}\label{kmikmir}
 K(t,\tau) := \sum_{p=1}^{2n+1} (-1)^p \, p \, \Tr\left( e^{-t \Delta_p (\tau)}\right).
\end{equation}
Let $\ker \Delta_p(\tau) = \{0\}$, then by Definition \ref{analtor}
\begin{equation}\label{naushki}
\log T_{\Ob}(\rho)=\frac{1}{2} \frac{d}{ds} \left.\left(   \frac{1}{\Gamma(s)} \int_0^\infty t^{s-1} K(t,\tau) dt \right)\right|_{s=0}.
\end{equation}
The right hand side of the formula is defined near $s=0$ by analytic continuation of the Mellin transform. Let $\w{E}_{\nu_p(\tau)} := G \times_{\nu_p(\tau)} \Lambda^p \mathfrak{p}^* \otimes V_\tau$, where $\nu_p(\tau):=\Lambda^p \Ad^* \otimes \tau: K \mapsto GL(\Lambda^p \mathfrak{p}^* \otimes V_\tau)$ and let $\w{\Delta}_p(\tau)$ be the lift of $\Delta_p(\tau)$ to $C^\infty(\HH^{2n+1}, \w{E}_{\nu_p(\tau)})$. Denote by $H_t^{\tau,p}:G \mapsto \End(\Lambda^p \mathfrak{p}^* \otimes V_\tau)$ the convolution kernel of $e^{-t \w{\Delta}_p(\tau)}$ as in \cite[p.~16]{MP111}. Let 
\begin{equation}\label{ananas}
 h_t^{\tau,p}(g):=\tr H_t^{\tau,p}(g),
\end{equation}
where $\tr$ denotes the trace in $\End(V_\nu)$. Put
$$ k_t^\tau(g) := \sum_{p=1}^{2n+1} (-1)^p  \, p \, h_t^{\tau,p}(g).$$
Then as in \cite[(4.15)]{MP111} we obtain from (\ref{kmikmir}) the following equality
 \begin{equation}\label{Selbergtrace}
 K(t,\tau)=\int_{\Gamma \bs G} \sum_{\gamma \in \Gamma} k_t^\tau(g^{-1} \gamma g) d \dot{g}.
 \end{equation}
\subsection{Selberg trace formula.} Collect the terms in the right hand side of (\ref{Selbergtrace}) according to their conjugacy classes:
\begin{equation}\label{heydomino}
K(t,\tau)=\sum_{\{\gamma \}} \int_{G_\gamma \bs G}  k_t^\tau(g^{-1} \gamma g) d \dot{g}.
\end{equation}
Above, $\{ \gamma \}$ denotes the conjugacy class of $\gamma$; $G_\gamma$ is the centralizer of $\gamma$ in $G$. Decompose (\ref{heydomino}) as a sum of  three terms:
\begin{equation}\label{sagihrdassichweine}
K(t,\tau)=I(t,\tau)+H(t,\tau)+H(t,\tau),
\end{equation}
where
\begin{equation}\label{brodski}
\begin{gathered}
I(t,\tau):=\vol(\Ob) k_t^\tau(e),\\
H(t,\tau):=\sum_{\{\gamma\}\,\textnormal{hyperbolic}} \int_{G_\gamma \bs G} k_t^\tau(g^{-1} \gamma g) d\dot{g},\\
E(t,\tau):=\sum_{\{\gamma\}\,\textnormal{elliptic}} \int_{G_\gamma \bs G} k_t^\tau(g^{-1} \gamma g) d\dot{g}.
\end{gathered}
\end{equation}
There exists an even polynomial $P_\sigma(i \lambda)$ such that \cite[Theorem 13.2]{Kna} 
\begin{equation}\label{icont}
 I(t,\tau) = \vol(\Ob) \sum_{\sigma \in \widehat{M}} \int_\R P_\sigma(i \lambda) \Theta_{\sigma,\lambda} (k_t^\tau) d \lambda,
\end{equation}
where $\Theta_{\sigma,\lambda}$ is the character of the representation $\pi_{\sigma,\lambda}$ as in Subsection \ref{prince}. There also exist even polynomials  $P_\sigma^\gamma(i \nu)$ such that \cite{Fe1} 
\begin{equation}\label{econt}
 E(t,\tau) = \sum_{\{\gamma   \} \, \textnormal{elliptic}} \vol(\Gamma_\gamma \bs G_\gamma) \sum_{\sigma \in \widehat{M}} \int_\R P^\gamma_\sigma(i \lambda) \Theta_{\sigma,\lambda} (k_t^\tau) d \lambda,
\end{equation}
where  $\Gamma_\gamma$ is the centralizer of $\gamma$ in  $\Gamma$. The sum in (\ref{icont}) and (\ref{econt}) is finite by the following proposition:
\begin{Lemma}\cite[Proposition 4.2]{MP111}\label{exitlight}
The character $\Theta_{\sigma, \lambda} (k^\tau_t)$ equals
\begin{equation*}
\Theta_{\sigma, \lambda} (k^\tau_t)=
\begin{cases}
e^{-t(\lambda^2 + \lambda_{\tau,k}^2)}, & \sigma\in \{ \sigma_{\tau,k}, w_0 \sigma_{\tau,k}\}  \textnormal{ for some } k,\\
0, & \textnormal{otherwise}.
\end{cases}
\end{equation*}
Above,  $w_0$ is as in Subsection \ref{liealg},
\begin{equation}\label{kittypuffy}
\lambda_{\tau, k} := \tau_{k+1}+n-k,
\end{equation}
$\tau_{k+1}$ and $\sigma_{\tau,k}$ are as in Definition \ref{dertttt}.
\end{Lemma}
\begin{remark}\label{entersandman}
The polynomials $P_\sigma^\gamma(i \nu)$ and $P_\sigma(i \nu)$ are invariant under the action of $w_o \in W(A)$:
$$ P_\sigma(i \nu) = P_{w_0 \sigma} (i \nu), \quad  P^\gamma_\sigma(i \nu) = P^\gamma_{w_0 \sigma} (i \nu).$$
\end{remark}

\subsection{$L^2$-torsion.} Let $\widetilde{\Delta}_p(\tau)$ be as in Subsection \ref{senbernary} and $\tau(m)$ be as in Definition \ref{mimirep}.
Denote by $\Tr_\Gamma (e^{-t \widetilde{\Delta}_p(\tau(m))})$ the $\Gamma$-trace of $e^{-t \widetilde{\Delta}_p(\tau(m))}$ on $\HH^{2n+1}$ as in \cite{lo}. It follows that 
$$ \textnormal{Tr}_\Gamma \, e^{-t \widetilde{\Delta}_p(\tau(m))} = \vol(\Ob) h_t^{\tau(m),p}(1),$$
where $h_t^{\tau(m),p}$ is as in (\ref{ananas}); see \cite{MP2}. Let $I(t,\tau(m))$ be as in (\ref{brodski}), then
$$\sum_{p=1}^{2n+1} \textnormal{Tr}_\Gamma \, e^{-t \widetilde{\Delta}_p(\tau(m))} = I(t,\tau(m)).$$
\begin{definition}
The $L^2$-torsion is defined by
$$ \log T_\Ob^{(2)} (\tau(m)) := \frac{1}{2}\frac{d}{ds} \left(\left. \frac{1}{\Gamma(s)} \int_{\R} t^{s-1}   \sum_{p=1}^{2n+1} \textnormal{Tr}_\Gamma \, e^{-t \widetilde{\Delta}_p(\tau(m))}  \right)\right|_{s=0}$$
for sufficiently large $m > m_0$, where $m_0$ is as in \cite[Proposition 5.3]{MP2}.
\end{definition}
\section{Analytic torsion under metric variation}\label{searchedfor}

In this section we study the Ray-Singer analytic torsion $T_\Ob (h,g)$  of a compact odd-dimensional Riemannian orbifold 
$(\Ob, g)$ with odd-dimensional singular strata; see Definition \ref{analtorry}. The main goal is to establish the invariance of the analytic torsion under certain deformations of the metric $g$ which we will now specify. 

\subsection{Admissible deformations of the metric}
Consider an orbifold atlas consisting of charts $(\widetilde{U}_i, G_i, \pi_i)$ as in Definition~\ref{orbicharty}.
\begin{definition}
By a smooth family of metrics on $\Ob$ we mean a collection of $G_i$-invariant metrics~$g_i(u)$, $u\in[0,1]$ on the orbifold charts  $\widetilde{U}_i$, depending smoothly on $u$ and for each $u$ satisfying the gluing condition saying that they define some metric $g(u)$ on $\Ob$. 
\end{definition}

\begin{exmp}\label{exexgirl}
Let $g$ be a metric on an orbifold $\Ob$ and  $m \in \R$. The family of metrics $\lambda \cdot g$, $\lambda \in [1,m]$ is a smooth family of metrics. 
\end{exmp}

\subsection{Deformation of the analytic torsion.}

Let $\Ob$ be an orbifold equipped with a smooth family of metrics $g(u)$, $u\in[0,1]$, and let $E \to \Ob$ be a flat orbibundle equipped with a Hermitian metric~$h$. Denote by $\Delta(u)^k$ the Laplacian acting on $E$-valued $k$-forms of $\Ob$. 
 Note that for different values of~$u$ the Laplacians $\Delta(u)^k$ act on different Hilbert spaces 
$L^2 \Omega^k(u)(\Ob)$. However, we can identify these spaces by
a natural isometry 
$$ T(u) : L^2 \Omega^k(u)(\Ob) \mapsto L^2\Omega^k(0)(\Ob),$$
$$ T(u) : f(x) \mapsto \left( \frac{\det g(0)(x)}{\det g(u)(x)} \right)^\frac14 \cdot f(x).$$
The operator $T(u)$ maps $H^2\Omega^k(u)(\Ob)$ to $H^2\Omega^k(0)(\Ob)$, hence it is an isometry between $\dom \Delta(u)^k$ and $\dom \Delta(0)^k$.
\begin{remark}\label{unforgiven}
 Let $X$ be a Riemannian manifold with isolated conical singularities which is not an orbifold, and let $g(u)$ be a family of metrics on $X$. 
Consider the corresponding family of Laplace-Beltrami operators $\Delta(u)$. It may happen that for some families $g(u)$ 
the multiplication by $\left( \frac{\det g_0(x)}{\det g_u(x)} \right)^\frac14$
is not an isometry between $\dom \Delta(u)$ and $\dom \Delta(0)$, even 
taking into account the possibility of different self-adjoint extensions of Laplacians, see Appendix \ref{sect:AppendixA} for details. Still there exists a  class of metric deformations for which $T(u)$ is an isometry between the domains of Laplacians \cite{MuVe}.
\end{remark}

Define the self-adjoint operators 
\begin{equation}\label{eq:HKU}
 H_k(u) := T(u) \circ \Delta_k(u) \circ T(u)^{-1}
\end{equation}
 with the fixed domain $\dom H_k(u) = H^2 \Omega^k(0)(\Omega)$. 
To establish the invariance of the analytic torsion we follow \cite{MuVe},
the key steps of whose proof we repeat here for the reader's convenience. Though our case differs from the one considered there, 
the necessary properties of all operators remain the same.
Namely, from the semigroups properties it follows that for any $u, u_0 \in [0,1]$,
\begin{equation*}
 \frac{e^{-t H_k (u)} - e^{- t H_k(u_0)}}{u-u_0} = \int_0^t e^{-(t-s)H_k(u_0)} \cdot \frac{H_k(u_0) - H_k(u)}{u-u_0} \cdot e^{- s H_k(u_0)} ds.
\end{equation*}
Taking $u \to u_0$ gives
\begin{equation*}
 \frac{\partial}{\partial u} \left. \Tr e^{-t H_k(u)} \right|_{u=u_0} = 
-t \cdot   \Tr \left(   \left.  \frac{\partial H_k(u)}{\partial u}   \right|_{u=u_0} â€¦\cdot e^{- t H_k(u_0)} \right).
\end{equation*}
From (\ref{eq:HKU}) and the commutativity of bounded operators under the trace we get
\begin{equation*}
 \frac{\partial}{\partial u} \left. \Tr e^{-t H_k(u)} \right|_{u=u_0} = - t \cdot \Tr \left(   \left.  \frac{\partial \Delta_k(u)}{\partial u}   \right|_{u=u_0} â€¦\cdot e^{- t \Delta_k(u_0)} \right).\end{equation*}
\begin{remark}\label{rem1111}
The operator on the right hand side of the latter equation 
is well-defined: although $\Delta(u)$ have different domains for different $u$, 
the exponent is an smoothing operator.  
\end{remark}

Let $*_u$ denote the Hodge-star operator associated with $g(u)$, and put 
$\alpha_u^k := *^{-1}_u \cdot \frac{\partial *}{\partial u},$
where $k$ denotes the restriction to the forms of degree $k$. By the same arguments as in \cite[p.153]{RS}, 
\begin{equation}\label{eq:rasddds}
\sum_{k=0}^{\dim{\Ob}} (-1)^k \cdot  k \cdot  \Tr \left( \frac{\partial \Delta_k(u)}{\partial u} \cdot   e^{-t \Delta_k(u)} \right) = 
\frac{\partial}{\partial t} \sum_{k = 0}^{\dim \Ob} (-1)^k \cdot  \Tr\left(\alpha_u^k e^{- t \Delta_k(u)}\right).
\end{equation}
\begin{remark}
 The order of composition of the operators in (\ref{eq:rasddds}) is swapped as compared to \cite[p.153]{RS}. 
 The reason is the same as in Remark \ref{rem1111}.
\end{remark}
Now put 
\begin{equation}\label{eq:fus}
 f(u,s) := \frac{1}{2} \sum_{k=0}^{\dim \Ob} (-1)^k \cdot k \cdot \frac{1}{\Gamma(k)} \int_0^\infty t^{s-1} \cdot  \Tr  e^{-t \Delta_k(u)}  dt.
\end{equation}
By the exponential decay of the heat trace we can differentiate the right hand sight 
of (\ref{eq:fus}) with respect to $u$. Together with (\ref{eq:rasddds}) this gives
\begin{equation}\label{eq:ufus}
 \frac{\partial f(u, s)}{\partial u} = \frac{1}{2} \sum_{k=0}^{\dim \Ob} 
 (-1)^k \frac{1}{\Gamma(s)} \int_0^\infty t^s \frac{d}{dt}  \Tr(\alpha_u^k e^{- t \Delta_k(u)}) dt.
\end{equation}
Suppose that $\ker \Delta_k(u) = \{ 0 \}$ for all $k$ and $u$. Then by definition,
\begin{equation}\label{eq:deftau} 
\log T_\Ob(h,g(u)) = \left. \frac{\partial}{\partial s} \right|_{s=0} f(u,s).
\end{equation}
By the same arguments as in \cite{Fe1}
\begin{equation}\label{ew:HTA}
\Tr \alpha_u^k e^{-t \Delta^k(u)} \sim 
\sum_{k=0}^\infty c_k t^{-n/2+k}+ \sum_{N \subset \text{ s. strata}} \sum_{k=0}^\infty d_k t^{-\dim N/2+k}
\end{equation}
for some $c_k, d_k \in \R$ as $t \to 0$. The second sum is taken over the singular strata $N$ of $\Ob$.
Putting together (\ref{eq:ufus}), (\ref{eq:deftau}) and (\ref{ew:HTA}), we obtain the following statement:
\begin{Th}
Let $\Ob$ be a Riemannian orbifold, and let $g(u), u \in [0,1]$ be a family 
of metrics on~$\Ob$. Suppose the cohomology $H^j(\Ob; E)$ vanishes for all $j$. 
Furthermore, let $l_k(u)$ denote the constant term of the asymptotic expansion (\ref{ew:HTA}). Then
$$\frac{\partial}{\partial u} \log T_\Ob(h, g(u)) = -1/2 \sum_{q=0}^{\dim{\Ob}} (-1)^q\, l_q(u).$$
\end{Th}

\begin{Corollary}\label{cor:metricind}
Assume $\dim \Ob$ is odd and all its singular strata are 
odd-dimensional as well, for example $\Ob = \Gamma \bs \HH^{2n+1}$; see Proposition \ref{pony}. Then $l_q(u) = 0$, and $\log T_\Ob(h, g(u))$ 
does not depend on~$u$. 
\end{Corollary} 
We will only need Corollary \ref{cor:metricind} for the families of metrics  $g(u)$ from Example \ref{exexgirl}.

\section{Analytic and $L^2$-torsion}\label{iamgone}
In this section we establish the asymptotic behavior of the 
analytic and the $L^2$-torsion for odd-dimensional hyperbolic orbifolds $\Ob = \Gamma \bs \HH^{2n+1}$. We will refer to some technical lemmas from Section~\ref{sec:AppendC}. Our main results are the following two theorems.

\begin{Th}\label{th1111}
Let $\tau(m)$ be as in Definition \ref{mimirep}. Then the difference between the analytic torsion~$T_\Ob (\tau(m))$ and the $L^2$-torsion $T^{(2)}_\Ob (\tau(m))$ can be expressed as:
$$\log T_\Ob (\tau(m)) - \log T^{(2)}_\Ob (\tau(m)) =  PE(m) + O(e^{-c_1 m}), \quad m \to \infty $$
 where $PE(m)$  is a pseudopolynomial in $m$ of  
degree $\frac{d^2+d+2}{2}$ with $d$ being the 
maximal dimension of a singular stratum of $\Ob$.
\end{Th}

\begin{Th}\label{th11111}
We have that
$$ \log T^{(2)}_\Ob (\tau(m)) = PI(m) +  O(e^{-c_2 m}),  \quad m \to \infty,$$
where $PI(m)$ is a polynomial in $m$ of degree $\frac{n^2+n+2}{2}$, hence
by Theorem \ref{th1111}:
$$\log T_\Ob (\tau(m)) =  PI(m) + PE(m) + O(e^{-c_3  m}), \quad m \to \infty.$$
\end{Th}

\begin{definition}
By \cite[Proposition 5.5]{MP2} the Mellin transforms
$$ \int_0^\infty t^{s-1} I(t,\tau(m)) dt$$
and
$$ \int_0^\infty t^{s-1}  E(t,\tau(m)) dt$$
of $I(t,\tau(m))$ and $E(t,\tau(m))$, respectively, are meromorphic functions of $s \in \C$, which are regular at $s=0$. Denote by  $MI(\tau(m))$ and $ME(\tau(m)))$ their values at $s=0$.
\end{definition}

\begin{Lemma}\label{analtorasymp}
\begin{equation*}
 \log T_\Ob(\tau(m)) = \frac{1}{2} \left(  MI(\tau(m)) + ME(\tau(m))  \right) + O(e^{-c m})
\end{equation*}
as $m \to \infty$.
\end{Lemma}
\begin{proof}
Note that $\ker(\Delta^p(\rho(m))) = \{0\}$ \cite[Lemma~9.2]{Pf}, hence by Corollary  
\ref{cor:metricind} the analytic torsion $T_\Ob(\tau(m);h,g)$ 
is invariant under smooth deformation of the metric $g$. As in Example \ref{exexgirl} we can rescale the metric by $\sqrt{m}$ or, equivalently, replace $\Delta_p(\tau(m))$ by $\frac{1}{m} \Delta_p(\tau(m))$, 
so that (\ref{naushki})
becomes
\begin{equation}\label{eq:recallanalm}
\log T_\Ob (\tau(m)) = \frac{1}{2} \frac{d}{ds}\left. \left(  \frac{1}{\Gamma(s)} \int_0^\infty t^{s-1} K\left(\frac{t}{m}, \tau(m) \right) dt \right) \right|_{s=0}.
\end{equation}
Note that 
\begin{equation}\label{eq:simplify}
 \left. \frac{d}{ds} \left(  \frac{1}{\Gamma(s)} \int_1^\infty t^{s-1} K\left(\frac{t}{m}, \tau(m)\right) dt \right)  \right|_{s=0} = \int_1^\infty t^{-1} K\left(\frac{t}{m}, \tau(m)\right) dt.
\end{equation}
Using \cite[(8.4)]{MP2}, it follows that 
\begin{equation}\label{eq:asymp1}
\int_1^\infty t^{-1} K\left(\frac{t}{m}, \tau(m)\right) dt = O(e^{-m /8}), \quad m \to \infty.
\end{equation}
Putting together     (\ref{eq:recallanalm}), (\ref{eq:simplify}) and (\ref{eq:asymp1}), we obtain
\begin{equation}\label{eq:TX}
 \begin{gathered}
  \log T_\Ob(\tau(m)) = 
  \left. \frac{1}{2} \frac{d}{ds} \left( \frac{1}{\Gamma(s)} \int_0^1 t^{s-1} K\left(\frac{t}{m}, \tau(m)\right) d t  \right) \right|_{s=0} + O(e^{-m/8}).
 \end{gathered}
\end{equation}
We need to estimate $K(t, \tau)$ for $0 < t \le 1$. By (\ref{sagihrdassichweine}),
\begin{equation}\label{eq:STF}
 K\left(\frac{t}{m}, \tau(m)\right) = I\left(\frac{t}{m}, \tau(m)\right) + E\left(\frac{t}{m}, \tau(m)\right) + H\left(\frac{t}{m}, \tau(m)\right).
\end{equation}
By \cite[(8,6)]{MP111} the contribution from $H\left(\frac{t}{m}, \tau(m)\right)$
to the first summand of the right hand side of (\ref{eq:TX}) decays exponentially as $m \to \infty$:
\begin{equation*}
 \frac{d}{ds} \left. \left(  \frac{1}{\Gamma(s)} \int_0^1 t^{s-1} 
 H\left(\frac{t}{m}, \tau(m)\right) dt \right) \right|_{s=0} = O(e^{-c_2 m}).
\end{equation*}
From (\ref{icont}), (\ref{econt}), Remark \ref{entersandman} and  Lemma \ref{exitlight} we obtain
\begin{equation}\label{est}
 \begin{gathered}
  I(t, \tau(m)) = 2 \, \vol(\Ob) \sum_{k=0}^n (-1)^{k+1} e^{- t \lambda^2_{\tau(m), k}} \int_\R e^{-t \lambda^2} P_{\sigma_{\tau(m),k}} (i \lambda) d \lambda, \\
  E(t, \tau(m)) = 2 \sum_{\{ \gamma\} \text{ elliptic}} \vol(\Gamma_\gamma \bs G_\gamma) \sum_{k=0}^n (-1)^{k+1} e^{- t \lambda^2_{\tau(m), k}} \int_\R e^{-t \lambda^2} P^\gamma_{\sigma_{\tau(m),k}} (i \lambda) d \lambda.
 \end{gathered}
\end{equation}

For further estimates of (\ref{eq:TX}) we need 
\begin{Lemma}
 There exists $C > 0$ such that
 \begin{equation}\label{unforgivenI}
 \int_0^1 t^{-1} I \left( \frac{t}{m} , \tau(m) \right) dt = \int_0^\infty t^{-1} I \left( \frac{t}{m} , \tau(m) \right) dt + O(e^{-Cm}),
 \end{equation}
 \begin{equation}\label{unforgivenII}
  \int_0^1 t^{-1} E \left( \frac{t}{m} , \tau(m) \right) dt = \int_0^\infty t^{-1} E \left( \frac{t}{m} , \tau(m) \right) dt + O(e^{-Cm}),
 \end{equation} 
 as $m \to \infty$. 
 \end{Lemma}
\begin{proof}
We repeat the estimate of (\ref{unforgivenI}) for the reader's convenience \cite[p.~24]{MP111}. Recall that the polynomial
\begin{equation*}
 P_{\sigma(m), k}(t) = \sum_{i=0}^n a_{k,i}(m) t^{2 i}
\end{equation*}
and there exists $C > 0$ such that
\begin{equation*}
|a_{k,i}| \le C m^{2n + n(n+1)/2} 
\end{equation*}
for all $k, i=0, \ldots ,n$ and $m \in \N$. Applying this estimate to (\ref{est}) and using that $\lambda_{\tau(m),i} \ge m$ for $i=0, \ldots , m$ we get
\begin{equation}\label{est1}
 \left|I\left(\frac{t}{m}, \tau(m)\right)\right| \le C e^{-c(m + t)}, \quad t \ge 1
\end{equation}
for some $c > 0$ and hence
\begin{equation} \label{114}
 \int_1^\infty t^{-1} I\left(\frac{t}{m}, \tau(m)\right) dt = O(e^{-cm}), \quad m \to \infty.
\end{equation}
By Lemma \ref{lemma:asymp3} we get the same estimates for $\int_1^\infty t^{-1} E\left(\frac{t}{m}, \tau(m)\right) dt$. 
\end{proof}
Using \cite[Proposition 5]{MP111} and (\ref{114}), we obtain Lemma \ref{analtorasymp}.
\end{proof}

\begin{Lemma}\label{l2torsasymp}\cite[(5.15)]{MP111}
We have that $\log T_\Ob^{(2)}(\tau(m)) = \frac{1}{2} \left(  MI(\tau(m))  \right)$. \qed
\end{Lemma}
\begin{Th}\label{ipryachutslyozy}
We have that
$$MI(\tau(m)) = PI(m), \quad ME(\tau(m)) = PE(m),$$
where $PI(m)$ is the polynomial in $m$ of  degree $\frac{n^2+n+2}{2}$, $PE(m)$ 
is the pseudopolynomial in $m$ of degree $\frac{d^2+d+2}{2}$ with $d$ being the 
maximal dimension of a singular stratum of $\Ob$.
\end{Th}
\begin{proof}
Note that
\begin{equation*}
MI(\tau(m))  = \sum_{k=0}^n (-1)^k \int_0^{\lambda_{\tau(m),k}} \PK dt,
\end{equation*}
\begin{equation*}
ME(\tau(m)) = \sum_{k=0}^n \sum_{\{\gamma\} \text{ elliptic}} (-1)^k \vol(\Gamma_\gamma \bs G_\gamma) \int_0^{\lambda_{\tau(m),k}} \PG dt. 
\end{equation*}
We estimate $ME(\tau(m))$ as in \cite[Corollary 5.7]{MP2}:
\begin{equation*}
 \sum_{k=0}^n (-1)^k \int_0^{\lambda_{\tau(m),k}} \PK dt = c(n) m \dim \tau(m) + O(m^{\frac{n(n+1)}{2}}),
\end{equation*}
where $c(n)$ is as in \cite[(2.24)]{MP2}. It remains to estimate $\sum_{k=0}^n (-1)^k \int_0^{\lambda_{\tau(m),k}} \PG dt.$ Let $\gamma$ be an elliptic element as in 
(\ref{pelmeni}). Recall that by definition of  $\lambda_{\tau(m),k}$
\begin{equation}\label{pringles}
\lambda_{\tau(m),0} > \lambda_{\tau(m),1} > \ldots > \lambda_{\tau(m),n}.
\end{equation}
Split the integrals 
\begin{equation}\label{XXX}
 \begin{gathered}
  \sum_{k=0}^n (-1)^k \int_0^{\lambda_{\tau(m),k}} \PG dt = \int_0^{\lambda_{\tau(m),n}} \sum_{k=0}^n (-1)^k  \PG dt + \sum_{k=0}^n (-1)^k \int_{\lambda_{\tau(m),n}}^{\lambda_{\tau(m),k}} \PG dt.
\end{gathered}
\end{equation}
We calculate the first summand in (\ref{XXX}).  By Lemma \ref{lemma:asymp1} it follows that $\sum_{k=0}^n (-1)^k  \PG$ does not depend on $t$, hence
\begin{equation}\label{XXXX}
  \sum_{k=0}^n (-1)^k \int_0^{\lambda_{\tau(m),n}} \PG dt = \lambda_{\tau(m),n} \cdot \sum_{k=0}^{n-1} \PG = (\tau_{n+1} + m) \cdot \sum_{k=0}^{n-1} \PG.
\end{equation}
The second equality is due to (\ref{kittypuffy}). Hence by Lemma \ref{lemma:asymp2} the 
expression in (\ref{XXXX}) is a pseudopolynomial of order $\leqslant \frac{d^2-d+2}{2}$. We are left to understand $\int_{\lambda_{\tau(m),n}}^{\lambda_{\tau(m),k}} \PG dt$. By Lemma \ref{lemma:asymp4} 
$$ \left|  m^{d(d-1)/2}   \int_{\lambda_{\tau(m),n}}^{\lambda_{\tau(m),k}}  \frac{\PG}{m^{d(d-1)/2}} dt \right| \leqslant m^{d(d+1)/2} \cdot (\lambda_{\tau(m),k} - \lambda_{\tau(m),n}) \cdot O(1) = O(m^{(d^2+d+2)/2}),$$
that proves the theorem.
\end{proof}

Theorems \ref{th1111} and \ref{th11111} follow from Lemmas \ref{analtorasymp}, \ref{l2torsasymp} and Theorem \ref{ipryachutslyozy}.

\section{Some technical lemmas}\label{sec:AppendC}
In this section we collect the technical lemmas which have been used in
 the proof of Theorem~\ref{ipryachutslyozy}, namely Lemmas~\ref{lemma:asymp1}, \ref{lemma:asymp2}, \ref{lemma:asymp3} and \ref{lemma:asymp4}. First, we need an explicit form of polynomials $P_\sigma^\gamma(\nu)$. For this, recall that the Killing form on $\mathfrak{g} \times \mathfrak{g}$ is defined by $B(X,Y)=\Tr (\ad(X) \circ \ad(Y))$.  For $\alpha \in \Delta^+(\mathfrak{g}_\C, \mathfrak{h}_\C)$ as in (\ref{mojbambino}) there exists a unique $H_\alpha' \in \mathfrak{g}_\C$ such that $B(H, H_\alpha') = \alpha(H)$ for all $H \in \mathfrak{g}_\C$. Denote $H_{\alpha} := \frac{2}{\alpha(H_\alpha')} H_\alpha'$. Let $\gamma \in \Gamma$ be as in (\ref{pelmeni}); without loss of generality assume that all $\varphi_i$ are different. Then the stabilizer $G_\gamma$ of $\gamma$ is equal to $\mathbb{T}^d \times SO_0(1,2d-1)$. The root system for $G_\gamma$ equals
$$ \Delta_{\gamma}(\mathfrak{g}_\C, \mathfrak{h}_\C) = \{\pm e_i \pm e_j,\, 1\le i < j \le d   \}.$$
We fix the positive system of roots by
$$ \Delta_\gamma^+ (\mathfrak{g}_\C, \mathfrak{h}_\C) = \{ \pm(e_i + e_j),\, 1\le i < j \le d   \}.$$
Let $W$ be the Weyl group of $\Delta(\mathfrak{m}_\C, \mathfrak{b}_\C)$ as in (\ref{nichegouzhenesvyazano}) and let the half-sum of positive roots $\rho_M$ be as in~(\ref{halfsummy}). For convenience, denote by $\Lambda$ the highest weight of a representation $\sigma \in \hat{M}$. 
\begin{Proposition}\cite{Fe1}
In the above notation,
\begin{equation}
\label{eq**}
  P_\sigma^\gamma (\nu) = 
\sum_{s \in W} \det(s) \prod_{\alpha \in \Delta_\gamma^+} \langle \alpha, -s(\Lambda + \rho_M) - i \nu e_1\rangle \zeta_{-s(\Lambda + \rho_M) - i \nu e_1} (\gamma).
\end{equation}
\end{Proposition}
For convenience we denote
\begin{equation}
\label{eq*}
 A(\Lambda, \nu) := \prod_{\alpha \in \Delta_\gamma^+} \langle \alpha, -\Lambda-i \nu e_1\rangle,   \quad B(\Lambda) = \zeta_{-\Lambda - i \nu e_1} (\gamma). 
\end{equation}
Note that if $\Lambda = v_2 e_2 + \ldots v_{n+1} e_{n+1}$, then (\ref{eq*}) becomes:
\begin{equation}\label{nedubravushka}
A(\Lambda, \nu) = \prod_{2\le j \le d} (- \nu^2 - v_j^2) \prod_{2 \le i < j \le d} (v_i^2-v_j^2), \quad  B(\Lambda) = e^{- i (v_{d+1} \phi_{d+1} + \ldots v_{n+1} \phi_{n+1})}.
\end{equation}
We see that $A(\Lambda, \nu)$ is an even polynomial in $\nu$ of order $2(d-1)$.
\begin{Lemma}\label{lemma:asymp1}
The expression $\sum_{k=0}^n (-1)^k P^\gamma_{\tau(m),k} (\nu)$ does not depend on $\nu$.
\end{Lemma}
\begin{remark}
 Compare with \cite[Corollary 9.9]{Pf} for the same result for $\sum_{k=0}^n (-1)^k P_{\tau(m),k} (\nu)$.
\end{remark}
\begin{remark}
 Every summand is a polynomial of order $2(d-1)$ in $\nu$, but the whole sum does not depend on $\nu$.
\end{remark}
\begin{Lemma}\label{lemma:asymp2}
The expression $\sum_{k=0}^n (-1)^k P^\gamma_{\tau(m),k} (\nu)$ is a pseudopolynomial in m of order $\le \frac{d(d-1)}{2}$.
\end{Lemma}
\begin{remark}
 Note that taking $\gamma = \id$ is equivalent to $d = n+1$. Then the sum $\sum_{k=0}^n (-1)^k P_{\tau(m),k} (\nu) = \sum_{k=0}^n (-1)^k P^\gamma_{\tau(m),k} (\nu)$ is a polynomial in $m$ of order $n(n+1)/2$,
which is asymptotically $\dim(\tau(m))$ as $m \to \infty$; compare \cite[Corollary~1.4]{MP2}.
\end{remark}
\begin{Lemma}\label{lemma:asymp3} Let
$$P^\gamma_{\tau(m),k} (\nu) = \sum_{i=0}^d a^\gamma_{k,i}(m) t^{2i},$$
then there exists $C > 0$ such that
$$ |a^\gamma_{k,i}| \leqslant C m^{2(d-1)+d(d-1)/2}.$$
\end{Lemma}

\begin{Lemma}\label{lemma:asymp4}
There exists $C_1>0$ such that for any $k = 0, \ldots, n$ and $\nu \in [\lambda_n, \lambda_k]$
$$P^\gamma_{\tau(m),k} (\nu) m^{-\frac{d(d-1)}{2}} < C_1.$$
\end{Lemma}
For brevity put $\lambda_i := \lambda_{\tau(m),i}$ for $i=0, \ldots, n$, where $\lambda_{\tau(m),i}$ is as in (\ref{kittypuffy}). Then by Definition~\ref{mimirep}, (\ref{halfsummy}) and (\ref{kittypuffy})
\begin{equation}
\label{eql1}
 \Lambda(\sigma_{\tau(m),k}) + \rho_M = \sum_{i=2}^{k+1} \lambda_{i-2} e_i + \sum_{i=k+2}^{n+1} \lambda_{i-1}e_i.
\end{equation}
Substituting  (\ref{eq*}) to (\ref{eq**}), we obtain
\begin{equation}
\label{eql2}
 \sum_{k=0}^n (-1)^k P^\gamma_{\tau(m),k} (\nu) = \sum_{k=0}^n \sum_{s \in W} (-1)^k \det(s) A(s\cdot \Lambda(\sigma_{\tau(m),k} + \rho_M), \nu) B( s \cdot \Lambda(\sigma_{\tau(m),k}) + \rho_M).
\end{equation}
\begin{proof}[Proof of Lemma \ref{lemma:asymp1}]
Collect the summands in the right hand side of (\ref{eql2}) into groups as in Definition~\ref{molly}; Lemma \ref{lemma:asymp1} follows from  Proposition~\ref{whiskeyinthejar} below, which states that each group separately does not depend on $\nu$.
\end{proof}
\begin{definition}\label{molly}
Fix some $(k,s)\in\{0,\ldots,n\}\times W$. Denote
$$U(k,s) := \left\lbrace (k',s') \in \{0,\ldots,n\}\times W \, \left| \, B( s \cdot ( \Lambda(\sigma_{\tau(m),k}) + \rho_M)) = B( s' \cdot( \Lambda(\sigma_{\tau(m),k'}) + \rho_M)) \right. \right\rbrace.$$

\end{definition}
\begin{Proposition}\label{whiskeyinthejar}
The sum 
$$ \sum_{(k',s') \in U(k,s)} A(s'\cdot ( \Lambda(\sigma_{\tau(m), k'}) + \rho_M), \nu)$$
does not depend on $\nu$. 
\end{Proposition}
\begin{proof}
Recall that $W$ acts by permutations and even number of sign changes. Note that if $w \in W$ acts by sign changes, then by (\ref{nedubravushka})
$$ A(\Lambda, \nu) = A(w \cdot \Lambda, \nu).$$
Hence it is sufficient to prove Proposition \ref{whiskeyinthejar} for $U(k,s) \cap \left( \{0,\ldots,n\} \times W' \right)$ instead of $U(k,s)$, where $W'$ is a subgroup of $W$ that acts by permutations.

Fix $k \in \{0,\ldots,n\}$, then $s(\Lambda(\sigma_{\tau(m),k})+\rho_M)$
spans $g_2 e_2 + \ldots + g_{n+1} e_{n+1}$ as $s \in W'$, where $(g_2, \ldots, g_{n+1})$ is an arbitrary permutation of the set 
$( \lambda_0, \lambda_1, \ldots, \widehat{\lambda}_k, \ldots, \lambda_n)$. Now fix some 
\begin{equation}\label{youngandbeautiful}
\lN{d+1}, \lN{d+2}, \ldots, \lN{n+1} \in \{ \lambda_0, \lambda_1, \ldots, \widehat{\lambda}_k, \ldots, \lambda_n \}
\end{equation}
with $\lN{i} \neq \lN{j}$ as $i \neq j$; without loss of generality assume that 
$\eta(d+1) < \eta(d+2) < \ldots < \eta(n+1)$. Then there exists $w \in W'$ and $k'\in [0, n]$ such that
\[B(-w(\Lambda(\sigma_{\tau(m),k'}) + \rho_M)) = e^{-i \left(\phi_{d+1} \lN{d+1} + \ldots + \phi_{n+1} \lN{n+1} \right)}.\]
By (\ref{nedubravushka}),
\begin{equation}\label{iwillloveyoutilltheend}
\begin{gathered}
U(w,k') = \left\lbrace s\in W', k\in[0,n]\, \left| \, s(\Lambda_{\sigma_{\tau(m),k}}+\rho_M)\right.\right. =\\ \left.\lambda_{\nu(2)} e_2 + \ldots + \lambda_{\nu(d)} e_d +  \lN{d+1} e_{d+1} + \ldots \lN{n+1} e_{n+1} \vphantom{s\in W', k\in[0,n]\,  s(\Lambda_{\sigma_{\tau(m),k}}+\rho_M) }  \right\rbrace,
\end{gathered}
\end{equation}
for some $\nu(i) \in \{ 0, 1, \ldots, \widehat{k}, \ldots, n\}$, $i \in [2,d]$ such that  $\nu(i) \neq \nu(i')$ for $i \neq i'$ and $\nu(i) \neq \eta(j)$ for any $j\in[d+1,n+1]$.
Note that if $(s,\kappa) \in U(w,k')$, then by (\ref{youngandbeautiful}) and (\ref{iwillloveyoutilltheend})
\begin{equation}\label{sticksandstones}
 \kappa \in K:= \{0, 1, \ldots, n \} \bs \cup_{i=d+1}^{n+1} \{\eta(i)\}.
\end{equation}
Let $(s_\kappa,\kappa)\in U(w,k)$. Put $\eta(d):=-1$ and $\eta(n+2):=n+2$. Let $\eta(d+B) < \kappa < \eta(d+B+1)$ for some $B=0, \ldots, n-d+1$.
\begin{definition}
For convenience set
$(\phi_2 | \phi_3 | \ldots | \phi_d) := \phi_2 e_2 + \ldots + \phi_d e_2$.
\end{definition}
By (\ref{iwillloveyoutilltheend}) 
$$s_\kappa (\Lambda_{\sigma_{\tau(m),\kappa}} + \rho_M) = \Lambda_\kappa +  \lN{d+1} e_{d+1} + \ldots \lN{n+1},$$
 where
\[ \Lambda_\kappa = (\lambda_0|\lambda_1| \ldots |\widehat{\lN{d+1}}|\ldots |\widehat{\lN{d+B}}| \ldots |\widehat{\kappa}|\ldots |\widehat{\lN{d+B+1}}| \ldots |\widehat{\lN{n+1}}|\ldots|\lambda_n),\]
and all its possible permutations. Note
\[A(\Lambda_\kappa+  \lN{d+1} e_{d+1} + \ldots \lN{n+1}, \nu) = \prod_{i \in K, i\neq k} (-\nu^2 - \lambda_i^2) \prod_{0 \le i < j  \le n, \, i,j \in K\bs \{k\} } (\lambda_i^2 - \lambda_j^2)\]
This is a polynomial of degree $2(d-1)$. We are interested in its values at the $2d$ points $\nu = \pm i \lambda_j$, $j\in K$:
\begin{equation}
\label{eqforA}
 \begin{gathered}
   A(\Lambda_\kappa, \pm i \lambda_j) = 0, \quad j \in K\bs \{\kappa\},\\
  A(\Lambda_\kappa, \pm i \lambda_\kappa) = (-1)^{k-B} \prod_{0 \le i < j \le n, \, i,j\in K} (\lambda_i^2 - \lambda_j^2).
 \end{gathered}
\end{equation}
It remains to compute the determinant of the permutation $s_\kappa$ that takes
\[ (\lambda_0|\lambda_1| \ldots |\widehat{\lN{d+1}}|\ldots |\widehat{\lN{d+B}}| \ldots |\widehat{\kappa}|\ldots |\widehat{\lN{d+B+1}}| \ldots |\widehat{\lN{n+1}}|\ldots|\lambda_n|\lN{d+1}|\ldots|\lN{n+1}   ),\]
to
\[ (\lambda_0|\lambda_1| \ldots | \widehat{\lambda_\kappa}| \ldots |\lambda_n ),\]
which equals
\begin{equation*}
 \begin{gathered}
  \det s_\kappa = (-1)^{d-\eta(d+1)} (-1)^{d-\eta(d+2)+1} \ldots (-1)^{d-\eta(d+B)+B-1} (-1)^{d-\eta(d+B+1)+B+1} \ldots (-1)^{d-\eta(n+1)+(n-d+1)}=\\
  (-1)^B (-1)^{d-\eta(d+1)+1} \cdot (-1)^{d-\eta(d+2)+2} \cdot \ldots \cdot (-1)^{d-\eta(n+1)+(n-d+1)} = (-1)^B \det s_0.
 \end{gathered}
\end{equation*}
Summing up over $\kappa$, we obtain
\begin{equation*}
 \begin{gathered}
  \sum_{\kappa \in K} \det(s_\kappa) (-1)^\kappa A(\Lambda_\kappa, \nu) = \sum_B \sum_{\kappa \in K \cap (\eta(d+B), \eta(d+B+1))} (-1)^\kappa (-1)^{B} \det(s_0) A(\Lambda_\kappa, \nu).
 \end{gathered}
\end{equation*}
This is a polynomial of order $2(d-1)$. Substituting $\nu = \pm i \lambda_{k'},  k' \in K$ and using (\ref{eqforA}), we get
\begin{equation*}
\begin{gathered}
 \sum_{\kappa \in K} \det(s_\kappa) (-1)^\kappa A(\Lambda_\kappa,  \pm i \lambda_{k'}) = (-1)^\kappa (-1)^B \det(s_0) (-1)^{k-B} \prod_{0 \le i < j \le n, \, i,j \in K} (\lambda_i^2 - \lambda_j^2) =\\
\det(s_0) \prod_{1 \le i < j \le n, \, i,j \in K} (\lambda^2_i - \lambda^2_j).
\end{gathered}
\end{equation*}
does not depend on $k'$, so the polynomial $\sum_{\kappa \in K} \det(s_\kappa) (-1)^\kappa A(\Lambda_\kappa, \nu)$ of order $2(d-1)$ has the same values at $2d$ points, hence it does not depend on $\nu$. All the possible permutations of $\Lambda_k$ are considered in the same way.
\end{proof}

\begin{proof}[Proof of Lemma \ref{lemma:asymp2}]
 By the proof of the previous lemma it suffices to consider
\begin{equation}\label{eq:A1}
 A(\Lambda_0, \pm i\lambda_0 )=\prod_{ 0 \le i < j \le n,\,  i,j\in K } (\lambda_i^2-\lambda_j^2).
\end{equation}
We need to estimate its order as $m \to \infty$. Recall that 
$$ \lambda_i = m + \tau_{i+1} + n - i,$$
then $\lambda_i^2 - \lambda_j^2$ has a linear growth in $m$. The set $K$ consists of $d$ elements, hence (\ref{eq:A1}) is the product 
of $\frac{d(d-1)}{2}$ factors of linear growth, hence
\begin{equation*}
 A(\Lambda_0, \pm i\lambda_0 )=O(m^{\frac{d(d-1)}{2}}), \quad m \to \infty.
\end{equation*}
\end{proof}

\begin{proof}[Proof of Lemma \ref{lemma:asymp3}]
Proceed with the same considerations as in Lemma \ref{lemma:asymp2}.
\end{proof}

\begin{proof}[Proof of Lemma \ref{lemma:asymp4}]
To prove the lemma we need to estimate 
$$ \frac{   \prod_{1\leqslant j \leqslant n, j \in K} (-\nu^2-\lambda_j^2) \prod_{1 \leqslant i < j \leqslant n, i,j \in K} (\lambda_i^2-\lambda_j^2)   }{m^{d(d-1)/2}}.$$
First note that 
$$ \frac{\prod_{1\leqslant i< j \leqslant n, i,j \in K} (\lambda_i^2-\lambda_j^2) }{ m^{\frac{(d-1)(d-2)}{2}  }   } =  \prod_{1\leqslant i< j \leqslant n, i,j \in K} \frac{(\lambda_i^2-\lambda_j^2)}{m} $$ 
is bounded for every $K$ and $m$. Second, 
$$ \frac{\prod_{1\leqslant j \leqslant n, j \in K} (-\nu^2-\lambda_j^2) }{ m^d  } =  \prod_{1\leqslant j \leqslant n, j \in K} \frac{(-\nu^2-\lambda_j^2)}{m}.$$ 
To estimate the latter note that
$$ \lambda_k^2 - \lambda_j^2 \leqslant \nu^2 - \lambda_j^2 \leqslant \lambda_k^2 - \lambda_j^2, \quad 1 \leqslant j \leqslant n, j \in K,$$
hence
$$\frac{|\nu^2-\lambda_j^2|}{m} \leqslant \frac{|\lambda_k^2 - \lambda_j^2| + |\lambda_k^2-\lambda_n^2|}{m}.$$
Note that this expression is bounded for every $j$ and $m$; this proves the lemma.
\end{proof}

\appendix 
\section{Deformation of a metric on the cone}\label{sect:AppendixA}
The goal of this appendix is to give an example to Remark \ref{unforgiven}.
 Let $g(u)$, $u \in [0,1]$ be a family of metrics  on the cone $M = \mathbb{S} \times [0,1]$, where $\mathbb{S}$ is a unit circle. As before, define
$$ T(u):f\mapsto f \cdot \sqrt{\rho}, \quad \sqrt{\rho}= \left( \frac{\det g(0)}{\det g(u)} \right)^{1/4}.$$
Consider the corresponding family of Laplace-Beltrami operators $\Delta(u)$; note that they are not essetially self-adjoint \cite{YCdeV}. The main 
goal of this section is to show that although $f \equiv 1$ obviously belongs to the domain of $\Delta(u)$, the function $[T(u)] f$ 
does not belong to the maximal extension of $\Delta(0)$, that is $H^2(M, g_0)$. To prove this, it is sufficient to show that $[T(u)]  f \not \in H^1(M, g_0)$ or, even weaker, that $||\nabla T(f)||_{L^2(M, g_0)} = \infty$.
Now let  $g(u) = dt^2 + l(u,t) d\phi^2$, $t \in [0,1]$ and $\phi \in \mathbb{S}$ for  arbitrary $l(u,t)$. Then 
$$\sqrt{\rho} = \left(   \frac{l(0,t)}{l(u,t)}    \right)^{1/4}$$
Let $f \equiv 1$, and suppose $l(u,t)$ does not depend on $\phi$, 
$$ |\nabla( \sqrt{\rho})|^2 = \left( \frac{\partial}{\partial t} \left( \frac{l(0,t)^{1/4}}{l(u,t)^{1/4}}\right) \right)^2, \quad ||\nabla Tf||_{L^2(M, g(0))} = \int_0^1 \left( \frac{\partial}{\partial t} \left( \frac{l(0,t)^{1/4}}{l(u,t)^{1/4}}\right) \right)^2 \cdot \big( l(0,t)\big)^{1/2} dt.$$
Let $l(u,t) = u \cdot t^{1/2} + t,$ then  
$$ ||\nabla Tf||_{L^2(M, g(0))} = \int_0^1 \frac{u^2}{64 t^{5/4} (\sqrt{t}+u)^{5/2}} dt = \infty.$$

\bibliography{foo}{}
\bibliographystyle{alpha}
\end{document}